\documentclass[letterpaper, reqno]{amsart}
\usepackage{inputenc}
\usepackage{amsmath, amssymb, amsthm,mathrsfs}
\usepackage{dsfont}
\usepackage{hyperref}
\usepackage[sort]{cite}
\usepackage{cleveref}
\usepackage{mathtools}

\DeclareUnicodeCharacter{03BC}{$\mu$}

\crefname{equation}{equation}{equations}
\Crefname{equation}{Equation}{Equations}

\newtheorem{theorem}{Theorem}
\newtheorem{prop}[theorem]{Proposition}
\newtheorem{lemma}[theorem]{Lemma}

\newtheorem{rmk}{Remark}

\newcommand{\E}{\mathbb{E}}
\newcommand{\bbT}{\mathbb{T}}
\newcommand{\bbR}{\mathbb{R}}

\newcommand{\bbZ}{\mathbb{Z}}
\newcommand{\bbOne}{\mathds{1}}
\newcommand{\rmI}{\mathrm{I}}
\newcommand{\rmII}{\mathrm{II}}
\newcommand{\rmIII}{\mathrm{III}}
\newcommand{\utrunc}{u_\textup{trunc}}

\newcommand{\jBra}[1]{\left\langle #1 \right\rangle}
\newcommand{\abs}[1]{\lvert #1 \rvert}
\newcommand{\lrabs}[1]{\left\lvert #1 \right\rvert}

\title[Wellposedness for periodic NLS with white noise dispersion]{On the wellposedness for periodic nonlinear {Schr\"odinger} equations with white noise dispersion}
\author{Gavin Stewart}
\address{Courant Institute of Mathematical Sciences, New York University, 251 Mercer St., New York, NY 10012, USA}
\address{Department of Mathematics, Rutgers University, 110 Frelinghuysen Rd., Piscataway, NJ, 08854, USA}
\email{\href{mailto:gavin.stewart@rutgers.edu}{gavin.stewart@rutgers.edu}}

\begin{document}

\begin{abstract}
    We study the wellposedness of the periodic nonlinear Schr\"odinger equation with white noise dispersion and a power nonlinearity given by
    \begin{equation*}
        idu = \Delta u \circ dW_t + \abs{u}^{p-1}u\;dt
    \end{equation*}
    We develop Strichartz estimates for this equation, which we then use to prove almost sure global wellposedness of this equation with $L^2$ initial data for nonlinearities with exponent $1 < p \leq 3$.  By generalizing the Fourier restriction spaces $X^{s,b}$ to the stochastic setting, we also prove that our solutions agree with the ones constructed by Chouk and Gubinelli (Commun. Part. Diff. Eq. 40(11):2047--2081, 2015) using rough path techniques.  We also consider the quintic equation ($p=5$), and show that it is analytically illposed in $L^1_\omega C_t L^2_x$. 
\end{abstract}

\maketitle

\section{Introduction}

We will study the nonlinear Schr\"odinger equation with white noise dispersion on the torus:
\begin{equation}\label{eqn:WNDNLS}
    \left\{\begin{array}{l}
        i du = \Delta u\,\circ dW_t + \abs{u}^{p-1} u\;dt\\
        u(t=0) = u_0
    \end{array}\right.
\end{equation}
This equation arises naturally as a scaling limit when considering equations of the form
\begin{equation}\label{eqn:smooth-mod-NLS}
    i \partial_t u = m(t) \Delta u + \abs{u}^{p-1}u
\end{equation}
when $m$ is a mean zero random process satisfying certain assumptions~\cite{martySplittingSchemeNonlinear2006,debussche1DQuinticNonlinear2011,cimpeanNonlinearSchrodingerEquation2019}.  The case when $p = 3$ is of particular interest, since in this case~\eqref{eqn:smooth-mod-NLS} models a signal traveling down an optical fiber, with $t$ representing the distance along the fiber, $x$ the retarded time, and $m$ the coefficient of chromatic dispersion~\cite{agrawalNonlinearFiberOptics2019,agrawalApplicationsNonlinearFiber2008}.  The case when $m$ is deterministic but mean zero is known as dispersion management.  Owing to their engineering applications, dispersion managed systems have received a great amount of attention over the past two decades, see the works~\cite{ablowitzDispersionManagementRandomly2004,turitsynTheoryAveragePulse1997,zharnitskyGroundStatesDispersionmanaged2000,zharnitskyStabilizingEffectsDispersion2001,abdullaevStableTwodimensionalDispersionmanaged2003,kunzeVariationalProblemLack2004,stanislavovaRegularityGroundState2005,hundertmarkDecayEstimatesSmoothness2009,greenExponentialDecayDispersionManaged2016,choiDispersionManagedNonlinear2021,choiThresholdsExistenceDispersion2017,choiWellposednessDispersionManaged2020,murphyModifiedScatteringDispersionmanaged2021,murphyWellposednessBlowupDispersionmanaged2021}  as well as the survey~\cite{turitsynDispersionmanagedSolitonsFibre2012} and references therein.  The problem of dispersion management with random $m$ was considered in~\cite{ablowitzDispersionManagementRandomly2004,malomedPropagationOpticalPulse2001}.

Equations of the form~\eqref{eqn:WNDNLS} were first considered by Marty in~\cite{martySplittingSchemeNonlinear2006} with Lipschitz nonlinearities, where solutions were shown to be globally wellposed.  To handle the non-Lipschitz nonlinearities, it is necessary to develop Strichartz estimates for~\eqref{eqn:WNDNLS}.  This was done on the real line by de Bouard and Debussche in~\cite{debouardNonlinearSchrodingerEquation2010} and later improved upon by Debussche and Tsutsumi in~\cite{debussche1DQuinticNonlinear2011}.  By combining these estimates with a truncation argument introduced in~\cite{debouardStochasticNonlinearSchrodinger1999, debouardStochasticNonlinearSchrodinger2003}, it is possible to prove that~\eqref{eqn:WNDNLS} is globally wellposed in $L^2$ on the real line for $1 < p \leq 5$.  This result can be generalized to quantum graphs: see~\cite{cimpeanNonlinearSchrodingerEquation2019}.  An alternative method for solving equations like~\eqref{eqn:smooth-mod-NLS} for rough modulations $m$ using rough path methods was given by Chouk and Gubinelli in~\cite{choukNonlinearPDEsModulated2015}, where it was used to solve~\eqref{eqn:WNDNLS} for $p=3$ on the real line and torus.  These results give global wellposedness wherever the deterministic equation is wellposed.  Numerical results by Belaouar, de Bouard, and Debussche in~\cite{belaouarNumericalAnalysisNonlinear2015} led the authors to conjecture that this wellposedness in fact holds for all $p < 9$, in agreement with the scaling of~\eqref{eqn:WNDNLS}.  This conjecture is supported by the further numerical calculations by Cohen and Dujardin~\cite{cohenExponentialIntegratorsNonlinear2017} and Laurent and Vilmart~\cite{laurentMultirevolutionIntegratorsDifferential2020}.  If this is indeed the case, it represents a regularization by noise phenomenon for the equation.  If we consider small initial data,  Dumont, Goubet, and Mammeri showed in~\cite{dumontDecaySolutionsOne2021} that for $p > 7$, solutions to the equation~\eqref{eqn:WNDNLS} with small initial data $u_0 \in L^1_x \cap H^1_x$ exist globally and satisfy the average decay estimate $\E \lVert u(t) \rVert_{L^\infty_x} \lesssim (1+t)^{-1/4}$.  There have also been several works considering numerical schemes for~\eqref{eqn:WNDNLS}~\cite{cohenExponentialIntegratorsNonlinear2017,cuiStochasticSymplecticMultisymplectic2017,laurentMultirevolutionIntegratorsDifferential2020,belaouarNumericalAnalysisNonlinear2015,martyLocalErrorSplitting2021}.

There seem to be few other results concerning equations with randomly modulated dispersion.  The results of Chouk and Gubinelli can be extended to cover KdV type equations~\cite{choukNonlinearPDEsModulated2014}, where it is shown the noise stabilizes these equations. Chen, Goubet, and Mammeri have studied a generalized long-wave equation with white noise dispersion~\cite{chenGeneralizedRegularizedLong2017} and showed that it is globally wellposed in $H^1$, and that localized solutions decay if the exponent of the nonlinearity is large.  Using a similar method, Dinvay proved in~\cite{dinvayStochasticBBMType2022} that solutions to a stochastic BBM equation exist globally in time.  We also mention the result in~\cite{buckmasterSurfaceQuasigeostrophicEquation2018}, where it is shown the the surface quasi-geostrophic equation with randomly modulated diffusion is globally well-posed with high probability for initial data in a suitable Gevrey space.

\subsection{Main results}

We will consider equation~\eqref{eqn:WNDNLS} on the torus $\bbT = \bbR / (2\pi\bbZ)$.  Using Strichartz estimates obtained for the linear propagator, we will prove that for $1 < p \leq 3$ equation~\eqref{eqn:WNDNLS}
is globally well-posed for $L^2$ initial data in the following sense:
\begin{theorem}\label{thm:gwp-theorem}
    Let $u_0 \in L^4_\omega L^2_x$, $1 < p \leq 3$.  Then,~\eqref{eqn:WNDNLS} admits a unique global solution $u \in L^4_\omega L^4_{t,\textup{loc}}L^4_x$.  Moreover, this solution $u$ is almost surely continuous in $L^2$ with conserved $L^2$ norm.
\end{theorem}

Solutions to the cubic equation were previously obtained by Chouk and Gubinelli in~\cite{choukNonlinearPDEsModulated2015}.  We will show that these solutions agree with the ones we construct:
\begin{theorem}\label{thm:rough-path-comparison}
    The $L^2$ solutions to the cubic NLS equation with white noise modulation constructed in~\cite{choukNonlinearPDEsModulated2015} coincide with those constructed in~\Cref{thm:gwp-theorem}.
\end{theorem}

We also study the equation with (deterministic) $L^2$ critical exponent $p=5$.  In the deterministic case, it was shown by Kishimoto in~\cite{kishimotoRemarkPeriodicMass2014} that the periodic NLS with quintic nonlinearity is illposed in $L^2$ the sense that the data to solution map is not smooth in $C_tL^2_x$, and we show that this illposedness extends to the stochastic equation:

\begin{theorem}\label{thm:quintic-illposedness}
    Suppose the data to solution map $S$ sending $u_0$ to the solution $u$ of~\eqref{eqn:WNDNLS} exists for $p=5$, and that it is continuous as an operator from a neighborhood of $0$ in $L^2$ into $L^1_\omega C^0_T L^2_x$ for $T$ sufficiently small.  Then, $S$ is not $C^5$.
\end{theorem}
\Cref{thm:quintic-illposedness} implies that we cannot use a fixed point argument in any space embedded in $L^1_\omega C^0_T L^2_x$ to construct solutions to the quintic equation.  However, it does not preclude the possibility of using a fixed point argument to construct solutions in some space $X \not\subset L^1_\omega C^0_T L^2_x$ and then showing \textit{a posteriori} that these solutions also lie in $L^1_\omega C^0_T L^2_x$.  In particular, it does not rule out an argument like we use to prove~\Cref{thm:gwp-theorem} that constructs a (local) solution in an auxiliary space, and then shows that the constructed solution is continuous in $L^2$.  However, the obvious generalization of the Strichartz estimate used for~\Cref{thm:gwp-theorem} is too weak to be of help here (even for nonlinearities $3 < p < 5$).

\subsection{Plan of the proof and obstacles}

The proof of~\Cref{thm:gwp-theorem} depends on Strichartz estimates for the linear equation
\begin{equation*}
    \left\{\begin{array}{l}
        idu = \Delta u \circ dW_t + F(t)\;dt\\
        u(t=0) = u_0\\
    \end{array}\right.
\end{equation*}
In particular, we will prove that the solution $u$ to the linear equation satisfies
\begin{equation}\label{eqn:L4-Strichartz-combined}
    \left\lVert u \right\rVert_{L^4(\Omega \times [0,t] \times \bbT)} \leq C(T) \left(\lVert u_0 \rVert_{L^4(\Omega; L^2(\bbT))} + \lVert F \rVert_{L^4(\Omega; L^{4/3}([0,T] \times \bbT))}\right)
\end{equation}
provided that $u_0$ and $F$ are adapted to the Brownian filtration.  This is similar to the $L^4$ estimates introduced by Bourgain in~\cite{bourgainFourierTransformRestriction1993a} for deterministic Schr\"odinger equations on the torus.  To obtain the homogeneous estimate, we directly expand the integral for $\lVert u \rVert_{L^4_{\omega,T,x}}^4$ in terms of the Fourier coefficients of $u_0$.  The resonant interactions can immediately be bounded by $\lVert u_0 \rVert_{L^2}^4$, and we use the identity $\E e^{iaW_t} = e^{-\frac{a^2}{2} t}$ to control the non-resonant interactions.  The inhomogeneous estimate is more difficult: because of the requirement that $u_0$ be adapted to the filtration at time $0$, we cannot use duality to get useful estimates from the homogeneous estimate.  Instead, we argue directly by expanding the expression for the Duhamel integral.  By combining these estimates with a truncation argument similar to the ones used in~\cite{debouardStochasticNonlinearSchrodinger1999,debouardNonlinearSchrodingerEquation2010}, we can show that~\eqref{eqn:WNDNLS} is globally wellposed for $1 < p \leq 3$.

Next, we show that our solutions agree with the rough path solutions constructed in~\cite{choukNonlinearPDEsModulated2015}.  To do so, we take advantage of the higher regularity in time for the rough path solutions to place them in an $X^{s,b}$-type space, which we show embeds in $L^4_{\omega,T,x}$ using Strichartz estimates.

Finally, we consider the equation with a quintic nonlinearity.  Recall that in the deterministic setting, the $L^6$ Strichartz estimates on the torus lose regularity~\cite{bourgainFourierTransformRestriction1993a}, and so are unsuited to prove wellposedness in $L^2$.  In fact, the deterministic quintic equation is known to be analytically illposed in $L^2$~\cite{kishimotoRemarkPeriodicMass2014}, and the question of whether weaker wellposedness results hold remains an open problem.  Both the loss of regularity for the Strichartz estimates and the analytic illposedness depend, in some sense, on the fact that the number of resonant interactions grows too large at high frequencies.  Since the equation with white noise dispersion has the same resonance structure, we are able to prove parallel results.

The plan of the paper is as follows: In~\cref{sec:notation} we present notation and basic results we will use throughout the paper. In~\cref{sec:strichartz-estimates} we prove $L^4$ based Strichartz estimates which will be necessary for establishing wellposedness, and introduce the spaces $X^{s,b}_4$ which we will use to prove~\Cref{thm:rough-path-comparison}.  We will also prove an $L^6$ Strichartz estimate with a loss of regularity. In~\Cref{sec:gwp}, we show that the equation is globally well-posed for $1 < p \leq 3$, and that our solutions to the cubic equation agree with those constructed in~\cite{choukNonlinearPDEsModulated2015}. Finally, in~\Cref{sec:quintic-ill-posedness}, we prove that the quintic equation is illposed in $L^1_\omega C_T L^2_x$.

\section{\label{sec:notation} Notation and Preliminaries}

We will use $\bbT = \bbR / (2\pi \bbZ)$ to denote the torus of size $2\pi$.  The Fourier transform $\mathcal{F} : L^2(\bbT) \to \ell^2(\bbZ)$ is given by 
\begin{equation*}
    \mathcal{F}u(k) = \hat{u}(k) = \frac{1}{2\pi} \int_{\bbT} f(x) e^{-ikx}\;dx
\end{equation*} 
The inverse Fourier transform is then 
\begin{equation*}
    \mathcal{F}^{-1} a_k = \check{a_k}(x) = \sum_{k \in \bbZ} a_k e^{ikx}
\end{equation*}
The Littlewood-Paley projectors $P_N f$ and $P_{\leq N} f$ are Fourier multipliers defined by their Fourier transforms
\begin{align*}
    \widehat{P}_N(k) =& \begin{cases}
        \chi(k) & N = 1\\
        \chi\left(\frac{k}{N}\right) - \chi\left(\frac{2k}{N}\right) & N > 1
    \end{cases}\\
    \widehat{P}_{\leq N}(k) =& \chi\left(\frac{k}{N}\right)
\end{align*}
where $\chi:\bbR \to [0,1]$ is a bump function supported in $B_2(0)$ with $\chi \equiv 1$ in $B_1(0)$.

The Strichartz estimates we introduce will require us to work with Lebesgue spaces in probability, time, and space.  We will write $L^\rho_\omega L^r_x$ for $L^\rho(\Omega; L^r(\bbT_L))$ and $L^\rho_\omega L^q_T L^r_x$ for $L^\rho(\Omega; L^q_\mathcal{P}(0,T; L^r(\bbT_L)))$, where $L^p_\mathcal{P}$ is the space of predictable $L^p$ processes. Similarly, we will write $L^\rho_\omega L^q_{T,x}$ if $q = r$ and $L^p_{\omega,T,X}$ if $\rho=q=r$.

To define and work with the $X^{s,b}_4(T)$ spaces in~\Cref{sec:strichartz-estimates}, we will need to use Sobolev spaces.  The $L^2$ based Sobolev spaces are defined for $0 < s < 1$ using the (equivalent) norms
\begin{equation*}
    \lVert f \rVert_{H^s(I;B)} = \lVert \jBra{D}^s f \rVert_{L^2(I;B)} \sim \lVert f \rVert_{L^2(I;B)} + \int_{I \times I} \frac{\lVert f(x) - f(y)\rVert_B}{\abs{x - y}^{\frac{1}{2} + s}}\;dxdy
\end{equation*}
where $B$ is some Banach space and $I$ is an interval.  These spaces satisfy the Sobolev inequality
\begin{equation*}
    \lVert f \rVert_{L^p} \lesssim_{s,p} \lVert f \rVert_{H^s}
\end{equation*}
for all $p$ with $\frac{1}{2} - s \leq \frac{1}{p} \leq \frac{1}{2}$.  At certain points in our argument, we will also need to work with the Besov spaces $B^{s}_{p,q}(I)$.  For $s \in (0,1)$, these spaces can be defined using the norm
\begin{equation*}
    \lVert f \rVert_{B^{s}_{p,q}(\Omega;B)} := \lVert f \rVert_{L^p(I; B)} + \left\lVert  \frac{\lVert \Delta_h f(x) \rVert_{L^p_x(I \cap (I-h);B)}}{h^{1+s}} \right\rVert_{L^q}
\end{equation*}
where $\Delta_h$ is the finite difference operator $\Delta_h f(x) = f(x) - f(x-h)$.
By examining the definition, we see that $B^s_{2,2} = H^s$.  Thus, by the standard embedding results for Besov spaces (see~\cite{bahouriFourierAnalysisNonlinear2011}), we have
\begin{equation*}
    \lVert f \rVert_{B^{s_2}_{p,q}} \leq \lVert f \rVert_{H^{s_1}}
\end{equation*}
whenever $p \geq 2$ and either $s_2 - \frac{1}{p} \leq s_1 - \frac{1}{2}$, $q\geq 2$, or $s_2 - \frac{1}{p} < s_1 - \frac{1}{2}$, $q \in [1,\infty]$.

Finally, we record the following convolution estimate for later use:
\begin{lemma}\label{thm:convo-lemma}
    For $f_n, g_n, h_n \in \ell^2(\bbZ)$ and $r_{n,m} \in \ell^2(\bbZ^2)$, we have the estimate
    \begin{equation*}
        \left\lVert \sum_{n,\ell} f_{k-n} g_{k-n-\ell} h_{k-\ell} r_{n,\ell} \right\rVert_{\ell^2_k} \leq \lVert f_n \rVert_{\ell^2} \lVert g_n \rVert_{\ell^2} \lVert h_n \rVert_{\ell^2} \lVert r_{n,m} \rVert_{\ell^2}
    \end{equation*}
\end{lemma}
\begin{proof}
    Let $\phi_n \in \ell^2(\bbZ)$.  Then, by Cauchy-Schwarz,
    \begin{equation*}\begin{split}
        \lrabs{\biggl\langle \phi_k, \sum_{n,\ell} f_{k-n} g_{k-n-\ell} h_{k-\ell} r_{n,\ell} \biggr\rangle} =\hspace{-3.75em}&\hspace{3.75em} \lrabs{ \sum_{k,n,\ell} \phi_k f_{k-n} g_{k-n-\ell} h_{k-\ell} r_{n,\ell} }\\
        \leq& \left(\sum_{k,n,\ell} \abs{ r_{n,\ell}\rvert^2 \lvert g_{k-n-\ell}}^2 \right)^{\frac{1}{2}} \left(\sum_{k,n,\ell} \abs{\phi_k}^2 \abs{ f_{k-n}}^2 \abs{h_{k-\ell}}^2\right)^{\frac{1}{2}}\\
        \leq& \lVert f_n \rVert_{\ell^2} \lVert g_n \rVert_{\ell^2} \lVert h_n \rVert_{\ell^2} \lVert r_{n,m} \rVert_{\ell^2} \lVert \phi_n \rVert_{\ell^2}
    \end{split}
    \end{equation*}
    which gives the desired result by duality.
\end{proof}

\section{Strichartz estimates\label{sec:strichartz-estimates}}

In this section, we will derive Strichartz estimates for the linear propagator of
\begin{equation}\label{eqn:linear-var-dispersion}
    \left\{\begin{array}{l}
    idu = \Delta u \; \circ dW_t\\
    u(t=s) = u_s
    \end{array}\right.
\end{equation}
which is given by $u(t) = e^{-i(W_t - W_s)\Delta} u_s$.  This operator can be represented in Fourier space as
\begin{equation}\label{eqn:prop-fourier-def}
    \widehat{e^{-i(W_t - W_s) \Delta}u_s} = e^{ik^2(W_t - W_s)} \hat{u}_s(k)
\end{equation}
see~\cite{martySplittingSchemeNonlinear2006,debouardNonlinearSchrodingerEquation2010}.  Since Brownian motion has stationary increments, it suffices to consider the case $s = 0$.  We begin with the homogeneous estimate.
\begin{theorem}\label{thm:strichartz-est-homog}
    For $u_0 \in L^4(\Omega; L^2(\bbT))$ independent of $\{W_t\}_{t > 0}$, we have that for any $\alpha \in \left(\frac{1}{2}, 2\right]$,
    \begin{equation*}
        \lVert e^{-iW_t \Delta} u_0 \rVert_{L^4_{\omega,T,x}} \lesssim_\alpha \left( T + T^{1 - \frac{\alpha}{2}} \right)^{1/4} \lVert u_0 \rVert_{L^4_\omega L^2_x}
    \end{equation*}
\end{theorem}
\begin{proof}
    Inverting the Fourier transform in~\eqref{eqn:prop-fourier-def} gives
    \begin{equation}\label{eqn:u-lin-def}
        u(t) := e^{-iW_t \Delta} u_0 = \sum_{n \in \bbZ_L}  a_k e^{i(kx + k^2 W_t)}
    \end{equation}
    where $a_k = \hat{u}_0(k)$.  Now, observe that $\lVert u \rVert_{L^4_{\omega,T,x}}^4 = \lVert (u \overline{u})^2 \rVert_{L^2_{\omega,T,x}}^2$.  Using~\cref{eqn:u-lin-def}, we find that
    \begin{equation*}
        (u \overline{u})^2 = \sum_{k \in \bbZ_L^4} a_{k_1} \overline{a_{k_2}} a_{k_3} \overline{a_{k_4}} e^{i(k_1 - k_2 + k_3 - k_4)x} e^{i(k_1^2 - k_2^2 + k_3^2 - k_4^2)W_t}
    \end{equation*}
    Thus, using the identity $\E e^{iaW_t} = e^{-a^2t/2}$ and fact that the $a_k$ are independent of the Brownian motion, we find that
    \begin{equation}\label{eqn:L4-division-homog}\begin{split}
        \E \int_{\bbT}(u \overline{u})^2\;dx =& 2\pi \E \sum_{k_1 - k_2 = k_4 - k_3} a_{k_1} \overline{a_{k_2}} a_{k_3} \overline{a_{k_4}} e^{-(k_1^2 - k_2^2 + k_3^2 - k_4^2)^2t / 2}\\
        \lesssim& \E \sum_{k_1,k_3} \abs{a_{k_1}}^2 \abs{a_{k_2}}^2 + \E \smashoperator[r]{\sum_{\substack{k_1 - k_2 = k_4 - k_3\\k_1 \neq k_2\\ k_1 \neq k_4}}} a_{k_1} \overline{a_{k_2}} a_{k_3} \overline{a_{k_4}} e^{-\frac{t(k_1^2 - k_2^2 + k_3^2 - k_4^2)^2}{2}}\\
        \lesssim& \lVert u_0\rVert_{L^4_\omega L^2_{x}}^4 + \E \sum_{k } \sum_{n,\ell \neq 0} a_k \overline{a_{k-n}} a_{k - n - \ell} \overline{a_{k - \ell}} e^{-2t n^2 \ell^2}
    \end{split}\end{equation}
    Integrating in time and applying Cauchy-Schwarz yields
    \begin{align*}
        \lVert u \rVert_{L^4_{\omega,T,x}}^4 \lesssim& T \lVert u_0 \rVert_{L^4_\omega L^2_x}^4 + \E \sum_{k}\sum_{n,\ell \neq 0} a_k \overline{a_{k-n}} a_{k - n - \ell} \overline{a_{k - \ell}} \frac{1 - e^{-2n^2 \ell^2T}}{2n^2\ell^2}\\
        \lesssim& \left(T  + \left(\sum_{n,\ell \neq 0} \frac{\left(1 - e^{-2 n^2\ell^2 T}\right)^2}{4n^4\ell^4}\right)^{1/2}\right) \lVert u_0 \rVert_{L^4_\omega L^2_x}^4
    \end{align*}
    By using the bounds $\frac{1 - e^{-2n^2\ell^2T}}{2n^2\ell^2} \leq T$ and $1 - e^{-2n^2\ell^2T} \leq 1$, we see that for $\frac{1}{2} < \alpha \leq 2$
    \begin{align*}
        \sum_{n,\ell \neq 0 } \frac{\left(1 - e^{-2 n^2\ell^2 T}\right)^2}{4n^4\ell^4}
        \lesssim&_\alpha T^{2-\alpha}
    \end{align*}
    so
    \begin{equation*}
        \lVert u \rVert_{L^4_{\omega,T,x}} \lesssim_\alpha \left(T + T^{1-\frac{\alpha}{2}}\right)^{1/4} \lVert u_0 \rVert_{L^4_\omega L^2_x}
    \end{equation*}
    as required.
\end{proof}

As explained in the introduction, we cannot obtain the inhomogeneous estimates directly from the homogeneous ones by duality in the stochastic setting.  Therefore, we must prove the inhomogeneous estimate via a more direct argument:

\begin{theorem}\label{thm:strichartz-est-inhomog}
    Let $f \in L^4(\Omega; L^{4/3}_\mathcal{P}( [0,T] \times \bbT))$.  Then, for $\frac{1}{2} < \alpha \leq 1$,
    \begin{equation*}
        \left\lVert \int_0^t e^{-i(W_t - W_s)\Delta} f(s) \;ds \right\rVert_{L^4_{\omega,T,x}} \lesssim_\alpha \left(T^{2} + T^{2-\alpha} \right)^{1/4} \lVert f\rVert_{L^4_\omega L^{4/3}_{T,x}}
    \end{equation*}
\end{theorem}
\begin{proof}
    Writing
    \begin{equation*}
        f(s,x) = \sum_k a_k(s) e^{ikx}
    \end{equation*}
    we find that
    \begin{equation}\label{eqn:L-4-estimate-splitting}\begin{split}
        \left\lVert \int_0^t e^{-i(W_t - W_s)\Delta} f(s) \;ds \right\rVert_{L^4_{{\omega,T,x}}} =& \int\limits_{[0,t]^4} \sum_{k,n,\ell} A_{k,n,\ell}(s) E_{k,n,\ell}(t,s)\;ds\\
        =& \int\limits_{[0,t]^4} \sum_k\sum_{n,\ell \neq 0} A_{k,n,\ell}(s) E_{k,n,\ell}  \;ds\\
        &+  \int\limits_{[0,t]^4} \sum_{k,n} A_{k,n,0}(s) E_{k,n,0}(t,s) \;ds\\
        &+  \int\limits_{[0,t]^4} \sum_{k,n} A_{k,0,n}(s) E_{k,0,n}(t,s) \;ds\\
        &-  \int\limits_{[0,t]^4} \sum_k A_{k,0,0}(s) E_{k,0,0}\;ds\\
        =:& \rmI + \rmII_1 + \rmII_2 - \rmIII
    \end{split}\end{equation}
    where
    \begin{equation*}
        A_{k,n,\ell}(s) = 2\pi a_k(s_1) \overline{a_{k-\ell}(s_2)} a_{k - \ell - n}(s_3) \overline{a_{k - n}(s_4)}
    \end{equation*}
    and
    \begin{equation*}
         E_{k,n,\ell}(t,s) = e^{i(W_t - W_{s_4}) (k-n)^2 - i(W_t - W_{s_3}) (k-\ell-n)^2 + i(W_t - W_{s_2}) (k-\ell)^2  -i(W_t - W_{s_1}) k^2}
    \end{equation*}
    For the last term, we use the Hausdorff-Young inequality to conclude that
    \begin{equation*}\begin{split}
        \abs{\mathrm{III}} \lesssim& \int_{[0,t]^4} \lVert a_k(s_1)\rVert_{\ell^4_k}\lVert a_k(s_2)\rVert_{\ell^4_k}\lVert a_k(s_3)\rVert_{\ell^4_k}\lVert a_k(s_4)\rVert_{\ell^4_k}\;ds\\
        \lesssim& t \lVert f(s) \rVert_{L^{4/3}_{T,x}}
    \end{split}\end{equation*}
    so
    \begin{equation}\label{eqn:inhomog-Str-III}
        \E \int_0^T \abs{\mathrm{III}} \;dt \lesssim T^2 \lVert f(s) \rVert_{L^4_\omega L^{4/3}_{T,x}}^4
    \end{equation}
    To handle the first term, we write $R_j(t) = \{s \in [0,t]^4 : s_j = \max_{i=1,\cdots,4}s_i\}$, and decompose $\rmI$ as
    \begin{equation}\label{eqn:L-4-I-split}
        \rmI = \rmI_1 + \rmI_2 + \rmI_3 + \rmI_4
    \end{equation}
    where 
    \begin{align*}
        \rmI_j(t) =& 2\pi \int_{R_j(t)}  \sum_k\sum_{n,\ell \neq 0} A_{k,n,\ell}(s) E_{k,n,\ell}(t,s) \;ds
        \end{align*}
    we will focus on $\rmI_1$ (the integrals over the other regions are analogous).  We have
    \begin{align*}
        \rmI_1(t) = 2\pi \int_{R_1(t)} \sum_k\sum_{n,\ell \neq 0}  A_{k,n,\ell}(s)  e^{2i(W_t - W_{s_1})\ell n} \tilde{E}_{k,n,\ell}(s) \;ds
    \end{align*}
    where $\tilde{E}_{k,n,\ell}(s)$ is a $\mathcal{F}_{s_1}$-measurable random variable with $\abs{\tilde{E}_{k,n,\ell}(s)} = 1$ a.s.  Since $e^{2i(W_t - W_{s_1})\ell n}$ is independent of the other factors in the integral, we can write
    \begin{equation}\label{eqn:inhomog-Str-I}\begin{split}
        \lrabs{\E \int_0^T \rmI_1 \; dt} \leq & 2\pi \E \int_0^T\int_{R_1(t)} \sum_k\sum_{n,\ell \neq 0}  \abs{A_{k,n,\ell}(s)} e^{-2(t - s_1)\ell^2 n^2} \;ds\;dt\\
        \lesssim& \E \sum_{n,\ell \neq 0} \frac{1 - e^{-2T\ell^2 n^2}}{n^2\ell^2} \sum_k \int_{[0,T]^4}\abs{A_{k,n,\ell}(s)}\;ds \\
        \lesssim_\alpha & T^{1 - \alpha} \E\lVert a_k(s) \rVert_{L^1_s \ell^4_k}^4\\
        \lesssim_\alpha & T^{2-\alpha} \lVert f(s) \rVert_{L^4_\omega L^{4/3}_{T,x}}^4
    \end{split}\end{equation}
    Finally, for the middle term, we use Plancherel's theorem to write
    \begin{equation*}\begin{split}
        \rmII_1 + \rmII_2 =& 4 \pi \left(\int_{[0,t]^2} \sum_k a_k(s_1) \overline{a_k}(s_2) e^{i(W_{s_1} - W_{s_2})k^2}\;ds\right)^2\\
        =& 16\pi \left(\Re \int_0^t \int_0^{s_1} \sum_k a_k(s_1) \overline{a_k}(s_2) e^{i(W_{s_1} - W_{s_2})k^2}\;ds_2 ds_1\right)^2\\
        =& 16\pi \left(\Re \int_0^t  \int_{\bbT_L} f(s_1,x) \int_0^{s_1} \overline{e^{-i(W_{s_1} - W_{s_2})\Delta} f(s_2, x)}\;ds_2 dx ds_1\right)^2
    \end{split}\end{equation*}
    By applying H\"older's inequality, we conclude that
    \begin{equation}\label{eqn:inhomog-Str-II}\begin{split}
        \E \int_0^T \abs{\rmII_1 + \rmII_2}\; dt \lesssim& T \lVert f \rVert_{L^4_\omega L^{4/3}_{T,x}}^2 \left\lVert \int_0^t e^{-i(W_t - W_s)\Delta} f(s) \;ds \right\rVert_{L^4_{\omega,T,x}}^2
    \end{split}
    \end{equation}
    Combining~\cref{eqn:inhomog-Str-I,eqn:inhomog-Str-II,eqn:inhomog-Str-III}, we see that
    \begin{equation*}\begin{split}
        \left\lVert \int_0^t e^{-i(W_t - W_s)\Delta} f(s)\;ds \right\rVert_{L^4_{\mathrlap{\omega,T,x}}}^4 \lesssim_\alpha& \left( T^{2-\alpha} + T^2\right) \lVert f \rVert_{L^4_\omega L^{4/3}_{T,x}}^4 \\
        &+ T \lVert f \rVert_{L^4_\omega L^{4/3}_{T,x}}^2 \left\lVert \int_0^t e^{-i(W_t - W_s)\Delta} f(s) \;ds \right\rVert_{L^4_{\omega,T,x}}^2
    \end{split}\end{equation*}
    which implies the desired result.
\end{proof}
\begin{rmk}
    If we consider the estimates on the torus $\bbT_L = \bbR / (2\pi L \bbZ)$ of length $2\pi L$, then the estimates in \Cref{thm:strichartz-est-homog,thm:strichartz-est-inhomog} become
    \begin{equation*}
        \lVert e^{-iW_t\Delta} u_0 \rVert_{L^4_{\omega,T,x}} \lesssim_\alpha \left( \frac{T}{L} + L^{3/2\alpha} \left(\frac{T}{L}\right)^{1 - \frac{\alpha}{2}} \right)^{1/4} \lVert u_s \rVert_{L^4_\omega L^2_x}^4
    \end{equation*}
    and
    \begin{equation*}
        \left\lVert \int_0^t e^{-i(W_t - W_s)\Delta} f(s) \;ds \right\rVert_{L^4_{\omega,T,x}} \lesssim_\alpha \left(\left(\frac{T}{L}\right)^{2} + L^{3\alpha} \left(\frac{T}{L}\right)^{2-\alpha} \right)^{1/4} \lVert f \rVert_{L^4_\omega L^{4/3}_{T,x}}
    \end{equation*}
    respectively.  If we formally set $\alpha = 1/2$ and let $L \to \infty$ while holding $T$ fixed, we recover the estimates obtained in~\cite{debouardNonlinearSchrodingerEquation2010} for the white noise dispersion Schr\"odinger equation set on $\bbR$.
\end{rmk}

Next, we show that the $L^4$ based Strichartz estimates can also be obtained for functions lying in an $X^{s,b}$-type space.  Define $X^{s,b}_4(T)$ be the space with the norm
\begin{equation*}
    \lVert e^{-iW_t \Delta} \psi \rVert_{X^{s,b}_4(T)} = \lVert \psi \rVert_{L^4_\omega H^b_t(0,T; H^s_x)}
\end{equation*}
Recall that the standard $X^{s,b}$ spaces are twisted versions of $H^b_tH^s_x$ (see~\cite[Section 2.6]{taoNonlinearDispersiveEquations2006}).  Thus,  we can see $X^{s,b}_4(T)$ as a natural generalization of the Fourier restriction spaces introduced by Bourgain in~\cite{bourgainFourierTransformRestriction1993a} to the stochastic setting.  Our decision to define the norm in the `untwisted' form is motivated by the fact that the Fourier transform is nonlocal, and hence breaks the causal structure given by the filtration.  We have the following result:
\begin{theorem}\label{thm:X-sb-Strichartz}
    If $u \in X^{0,b}_4(T)$ with $b > 5/16$, then $u \in L^4_\omega L^4_T L^4_x$ with
    \begin{equation} \label{eqn:L4-Xsb-bound}
        \lVert u \rVert_{L^4_\omega L^4_T L^4_x} \lesssim_b \lVert u \rVert_{X^{0,b}_4(T)}
    \end{equation}
\end{theorem}
\begin{proof}
    By using the nesting property of the $X^{0,b}_4$ spaces, we may assume without loss of generality that $b < 1/2$.  Let $u = e^{-iW_t \Delta} \psi$.  Then, we can write
    \begin{equation}\label{eqn:L4-X-sb-division}\begin{split}
        \lVert u \rVert_{L^4_{\omega,T,x}}^4 =& \lVert e^{-iW_t \Delta} \psi \rVert_{L^4_{\omega,T,x}}^4\\
        \leq& 2 \E \int_0^T \lVert e^{-iW_t\Delta} \psi(t,x) \rVert_{L^2_x}^4\;dt\\
        &- 2\pi \E \int_0^T \sum_k \abs{\hat{\psi}_k(t)}^4\;dt\\
        &+ 2\pi \E \int_0^T \sum_k \sum_{n,\ell \neq 0} e^{iW_t 2n\ell} \hat{\psi}_k(t)\overline{\hat{\psi}_{k-n}(t)}\hat{\psi}_{k-n-\ell}(t)\overline{\hat{\psi}_{k-\ell}(t)}
    \end{split}\end{equation}
    The second term on the right is manifestly negative and so can be ignored.  For the first term, the Sobolev embedding gives us that
    \begin{equation*}
        2 \E \int_0^T \lVert e^{-iW_t\Delta} \psi(t,x) \rVert_{L^2_x}^4\;dt \lesssim \lVert \psi \rVert_{L^4_\omega H^b_t(0, T; L^2_x)} = \lVert u \rVert_{X^{0,b}_4(T)}
    \end{equation*}
    Thus, it only remains to deal with the last term.  Let
    \begin{equation*}
        S_{n, \ell} = (n\ell)^{-\alpha}
    \end{equation*}
    where $\alpha \in \left(\frac{1}{8b - 2}, 2\right)$ (such an $\alpha$ exists by our assumption $b < 1/2$).  To ease notation, we will also write
    \begin{equation*}
        \Psi_{k,n,\ell}(t) = \hat{\psi}_k(t)\overline{\hat{\psi}_{k-n}(t)}\hat{\psi}_{k-n-\ell}(t)\overline{\hat{\psi}_{k-\ell}(t)}
    \end{equation*}
    With this setup, we can rewrite the last term in~\eqref{eqn:L4-X-sb-division} as
    \begin{equation*}\begin{split}
        \E \int_0^T \sum_k \sum_{n,\ell \neq 0} e^{iW_t 2n\ell} \Psi_{k,n,\ell}(t) =\hspace{-3em}&\hspace{3em} \E \int_0^T \sum_k \sum_{\substack{n,\ell \neq 0\\ 2S_{n,\ell} > t}} e^{2iW_t n\ell} \Psi_{k,n,\ell}(t)\;dt\\
        &+ \E \int_0^T \sum_k \sum_{\substack{n,\ell \neq 0\\ 2S_{n,\ell} \leq t}} \frac{e^{2iW_t n\ell}}{S_{n,\ell}}\int_{S_{n,\ell}}^{2S_{n,\ell}}\Delta_s\Psi_{k,n,\ell}(t) \;ds dt\\
        &+ \E \int_0^T \sum_k \sum_{\substack{n,\ell \neq 0\\ 2S_{n,\ell} \leq t}}  \frac{e^{2iW_t n\ell}}{S_{n,\ell}}\int_{S_{n,\ell}}^{2S_{n,\ell}} \Psi_{k,n,\ell}(t-s) \;ds dt\\
        =:& \rmI + \rmII + \rmIII
    \end{split}\end{equation*}
    so it only remains to show that $\rmI$, $\rmII$ and $\rmIII$ can be controlled by $\lVert u \rVert_{X^{0,b}_4(T)}^4$.  For $\rmI$, we note that $S_{n,\ell} \leq 1$, so we can restriction the integral to $t \in [0,2]$.  Using~\Cref{thm:convo-lemma}, we then find that
    \begin{equation*}\begin{split}
        \abs{\rmI} =& \lrabs{ \E \int_0^2 \sum_k \sum_{\substack{n,\ell \neq 0\\ 2S_{n,\ell} > t}} e^{-2iW_t n\ell} \Psi_{k,n,\ell}(t)\;dt}\\
        \leq& \E \int_0^2 \bigg( \sum_{\substack{n,\ell \neq 0\\ 2S_{n,\ell} > t}} 1 \bigg)^{1/2} \lVert \psi(t) \rVert_{L^2_x}^4 \;dt\\
        \lesssim_\alpha& \E \int_0^2 t^{-\frac{1}{2\alpha}} \sqrt{-\log (t/2)} \lVert \psi(t) \rVert_{L^2_2}^4\;dt\\
        \lesssim_\alpha& \E \lVert \psi(t) \rVert_{L^{p^*}_tL^2_x}^4\\
        \lesssim_\alpha& \lVert u \rVert_{X^{0,b}_4(T)}^4
    \end{split}\end{equation*}
    where $\frac{1}{p^*} = \frac{1}{2} - b$.  

    We now turn to $\rmII$.  Notice that we can write
    \begin{equation*}\begin{split}
        \Delta_s\Psi_{k,n,\ell}(t) = & \sum_{j=1}^4 \Psi^{(j)}_{k,n,\ell}(t,t-s)
    \end{split}\end{equation*}
    where $\Psi^{(j)}_{k,n,\ell}(t, t-s)$ contains $j$ factors of the form $\Delta_s\hat{\psi}_\bullet(t)$ and $4-j$ factors of the form $\hat{\psi}_\bullet(t-s)$.  Thus, we have that
    \begin{equation*}
        \rmII = \sum_{j=1}^4 \rmII_j
    \end{equation*}
    with
    \begin{equation*}
        \rmII_j = \E \int_0^T \sum_k \sum_{\substack{n,\ell \neq 0\\ S_{n,\ell} \leq t}} e^{2iW_t n\ell} \frac{1}{S_{n,\ell}}\int_{S_{n,\ell}}^{2S_{n,\ell}}\Psi^{(j)}_{k,n,\ell}(t,t-s) \;ds dt
    \end{equation*}
    Using~\Cref{thm:convo-lemma}, we find that
    \begin{equation*}\begin{split}
        \abs{\rmII_j} \lesssim& \E \int_0^T \int_0^2 \Bigg( \sum_{\substack{n,\ell \neq 0\\2S_{n,\ell} > s > S_{n,\ell}}} \frac{1}{S_{n,\ell}^2}\Bigg)^{1/2} \lVert \Delta_s\psi(t)\rVert_{L^2_x}^j \lVert \psi(t-s) \rVert_{L^2_x}^{4-j} \;ds dt\\
        \lesssim_\alpha& \E \int_0^T \int_0^2 s^{-1- \frac{1}{2\alpha}} \sqrt{-\log (s/2)} \lVert \Delta_s\psi(t) \rVert_{L^2_x}^j \lVert \psi(t-s) \rVert_{L^2_x}^{4-j} \;ds dt\\
    \end{split}\end{equation*}
    For $j = 4$, we have that
    \begin{equation*}\begin{split}
        \rmII_4 \lesssim_\alpha& \E \int_0^T \int_0^2 s^{-1 - \frac{1}{2\alpha} + 4b} \sqrt{-\log (s/2)} \Bigg( \frac{\lVert \Delta_s\psi(t) \rVert_{L^2_x}}{s^b}\Bigg)^4\;dsdt\\
        \lesssim_\alpha& \E \lVert \psi(t) \rVert_{B^{b - \frac{1}{4}}_{4,4}([0,T]; L^2_x)}^4\\
        \lesssim_\alpha& \lVert u \rVert_{X^{0,b}_4(T)}^4
    \end{split}\end{equation*}
    For $j = 1,2,3$, let $\delta = \min (1/48, b - 5/16)$.  Then, for $\frac{1}{p} = \frac{3}{16} - \delta$, we can find $b_j \in (0, b]$ such that
    \begin{equation*}
        \frac{1}{4} + 4\delta \leq j b_j \leq \min\bigg( j\Big( \frac{5}{16} + \delta\Big), 1 - \frac{j}{p}\bigg)
    \end{equation*}
    which allows us to write
    \begin{equation*}\begin{split}
        \abs{\rmII_j} \lesssim_\alpha& \E \int_0^2 s^{-1 - \frac{1}{2\alpha} + j b_j} \sqrt{-\log (s/2)} \int_s^T \bigg( \frac{\lVert \Delta_s\psi(t) \rVert_{L^2_x}}{s^{b_j}} \bigg)^j \lVert \psi(t-s) \rVert_{L^2_x}^{4-j}\;dtds\\
        \lesssim_\alpha& \E \Bigg \lVert \int_s^T \bigg( \frac{\lVert \Delta_s\psi(t) \rVert_{L^2_x}}{s^{b_j}} \bigg)^j \lVert \psi(t-s) \rVert_{L^2_x}^{4-j}\;dt \Bigg \rVert_{L^{r_j'}_s}
    \end{split}\end{equation*}
    where $\frac{1}{r_j} = 2 - \frac{4}{p} - jb_j$.  Defining $\frac{1}{p_j} = \frac{1}{p} - b_j$, we then have that
    \begin{equation*}\begin{split}
        \abs{\rmII_j} \lesssim_\alpha& \E \sup_{h > 0} \Bigg \lVert \int_h^T \bigg( \frac{ \lVert \Delta_h \psi(t) \rVert_{L^2_x}}{s^{b_j}} \bigg)^{j} \lVert \psi(t-s) \rVert_{L^2_x}^{4-j}\;dt \Bigg \rVert_{L^{r_j'}_s}\\
        \lesssim_\alpha& \E \lVert \psi \rVert_{B^{b_j}_{p_j,\infty}([0,T]; L^2_x)}^j \lVert \Psi(t) \rVert_{L^p_TL^2_x}^{4-j}\\
        \lesssim_\alpha& \lVert u \rVert_{X^{0,b}_4(T)}^4
    \end{split}\end{equation*}
    completing the argument for $\rmII$.

    Finally, for $\rmIII$, we note that $e^{2i(W_t - W_s)n\ell}$ is independent of $e^{2iW_sn\ell} \Psi_{k,n,\ell}(t-s)$, so we have
    \begin{equation*}\begin{split}
        \abs{\rmIII} =& \lrabs{\E \int_0^T \sum_k \sum_{\substack{n,\ell\neq 0\\ 2S_{n,\ell} \geq t}} \frac{1}{S_{n,\ell}} \int_{S_{n,\ell}}^{2S_{n,\ell}} e^{-2s(n\ell)^2} e^{2iW_sn\ell} \Psi_{k,n,\ell}(t-s)\;ds dt}\\
        \leq& \E \int_0^T \sum_k \sum_{\substack{n,\ell\neq 0\\ 2S_{n,\ell} < 2t}} \frac{1}{S_{n,\ell}} \int_{S_{n,\ell}}^{2S_{n,\ell}} e^{-2s(n\ell)^2} \abs{\Psi_{k,n,\ell}(t-s)}\;ds dt\\
        \lesssim& \E\int_0^2 \int_s^T \Biggl( \sum_{\substack{n,\ell \neq 0\\ 2S_{n,\ell} > s}} \frac{e^{-4 S_{n,\ell} (n\ell)^2}}{S_{n,\ell}^2} \Biggr)^{1/2} \lVert \psi(t-s) \rVert_{L^2_x}^4\;ds dt
    \end{split}\end{equation*}
    By writing $N = n\ell$, we can rewrite the sum as
    \begin{equation*}\begin{split}
        \sum_{\substack{n,\ell \neq 0\\ 2S_{n,\ell} > s}} \frac{e^{-4 S_{n,\ell} (n\ell)^2}}{S_{n,\ell}^2}  \sim  \sum_{N=1}^{s^{-1/\alpha}} N^{2\alpha} e^{-4N^{2-\alpha}} \tau(N)
    \end{split}
    \end{equation*}
    where $\tau(N)$ is the number of divisors of $N$.  Since $\alpha < 2$, this sum converges as $s \to 0^+$, and hence
    \begin{equation*}\begin{split}
        \abs{\rmIII} \lesssim_\alpha& \E \int_0^2 \int_s^T \lVert \psi(t-s) \rVert_{L^2_x}^4\;dtds\\
        \lesssim_\alpha& \lVert u \rVert_{X^{0,b}_4(T)}^4
    \end{split}\end{equation*}
    completing the proof.
\end{proof}

Finally, we give an $L^6$ estimate for the homogeneous equation:

\begin{theorem}\label{thm:L6-strichartz}
    Let $u_0 \in L^6_\omega L^2_x$ be independent of $\{W_t\}_{t > 0}$.  Then, for $\alpha \in (1/2, 1]$ and  $\epsilon > 0$,
    \begin{equation*}
        \left\lVert P_{\leq N} e^{-iW_t\Delta} u_0 \right\rVert_{L^6_{\omega,T,x}} \lesssim_{\alpha,\epsilon} \left(T + T^{1-\alpha}\right)^{1/6}N^\epsilon \lVert u_0 \rVert_{L^6_\omega L^2_x}
    \end{equation*}
\end{theorem}
\begin{proof}
    Let us denote the $k$th Fourier coefficient of $P_{\leq N} u_0$ by $a_k$, so
    \begin{equation*}
        P_{\leq N} e^{-iW_t\Delta} u_0(x) = \sum_{\abs{k} \leq 2N} a_k e^{ikx + i W_t k^2}
    \end{equation*}
    Now, $\lVert e^{-iW_t \Delta} u_0 \rVert_{L^6}^6 = \lVert (e^{-iW_t \Delta} u_0)^3 \rVert_{L^2}^2$, and we can write
    \begin{equation*}\begin{split}
        \mathcal{F}\bigl( (e^{-iW_t \Delta} u_0)^3 \bigr)(k) =& \sum_{k_1, k_2} a_{k_1} a_{k_2} a_{k - k_1 - k_2} \exp(iW_t( k_1^2 + k_2 + (k-k_1 - k_2)^2))\\
        =& \sum_{j} \sum_{(k_1, k_2) \in S_{k,j}} a_{k_1} a_{k_2} a_{k - k_1 - k_2} \exp(iW_tj)
    \end{split}\end{equation*}
    where $S_{k,j}$ is the set
    \begin{equation*}
        S_{k,j} = \bigl\{(k_1, k_2) : k_1^2 + k_2^2 + (k - k_1 - k_2)^2 = j\bigr\}
    \end{equation*}
    Thus, we can use Plancherel's identity to write
    \begin{equation}\label{eqn:L6-expansion}\begin{split}
        \lVert e^{-iW_t \Delta} u_0 \rVert_{L^6_{\omega,T,x}}^6 =& 2\pi \E \int_0^T\sum_{k} \lrabs{ \sum_{j} \sum_{(k_1,k_2) \in S_{k,j}} a_{k_1} a_{k_2} a_{k - k_1 - k_2} \exp(i W_t j) }\;dt\\
        =& 2\pi  \E\int_0^T \sum_{r} \exp\left(-\frac{t}{2}r^2\right)\;dt \sum_{j,k} \sum_{\substack{(k_1,k_2) \in S_{k,j}\\(k_1',k_2') \in S_{k,j+r}}}A_{k,k_1,k_2,k_1',k_2'}\\ 
    \end{split}\end{equation}
    where 
    \begin{equation*}
        A_{k,k_1,k_2,k_1',k_2'} = a_{k_1} a_{k_2} a_{k - k_1 - k_2} \overline{a_{k_1'} a_{k_2'} a_{k - k_1' - k_2'}}
    \end{equation*}
    The innermost sum in the $j$ and $k$ variables can be bounded independent of $r$:
    \begin{equation*}\begin{split}
        \lrabs{\sum_{j,k} \sum_{\substack{(k_1,k_2) \in S_{k,j}\\\mathclap{(k_1',k_2') \in S_{k,j+r}}}}a_{k_1} a_{k_2} A_{k,k_1,k_2,k_1',k_2'} } \leq& \sup_{\abs{n} \lesssim N} \sum_j \sum_{\substack{(k_1, k_2) \in S_{n,j}\\\mathclap{(k_1', k_2') \in S_{n, j+r}}}} \abs{a_{k_1} a_{k_2} a_{k_1'} a_{k_2'}}\lVert a_k \rVert_{\ell^2}^2\\
        \leq& \sup_{\substack{\abs{n} \lesssim N\\ \abs{j},\abs{j+r} \lesssim N^2}}  \abs{S_{n,j}}^{1/2}\abs{S_{n,j+r}}^{1/2} \lVert a_k \rVert_{\ell^2_k}^6\\
        \lesssim_\epsilon & N^\epsilon \lVert u_0 \rVert_{L^2}^6
    \end{split}\end{equation*}
    where on the last line we have used the fact that $(k_1, k_2) \in S_{n,j}$ if and only if $(3k_1 + 3k_2 - 2k, k_1 - k_2)$ is a lattice point on the ellipse
    \begin{equation*}
        X^2 + 3Y^2 = t
    \end{equation*}
    with $t = 6j - 2k^2$, and there are $O_{\epsilon}(t^{3\epsilon}) = O_{\epsilon}(N^{6\epsilon})$ such points by a result of Bombieri and Pila~\cite{bombieriNumberIntegralPoints1989} (cf.~\cite{bourgainFourierTransformRestriction1993a}).  Separating out the resonant contribution at $r = 0$ from the other contributions, we see that the $L^6$ norm is given by
    \begin{equation*}\begin{split}
        \lVert e^{-iW_t \Delta} u_0 \rVert_{L^6_{\omega,T,x}}^6 \lesssim& \left(T + \sum_{r \neq 0} \frac{1 - e^{-\frac{T}{2}r^2}}{r^2}\right) N^{6\epsilon} \lVert u_0 \rVert_{L^6_\omega L^2_x}^6\\
        \lesssim_{\epsilon,\alpha}& \left(T + T^{1-\alpha}\right)N^{6\epsilon} \lVert u_0 \rVert_{L^6_\omega L^2_x}^6
    \end{split}\end{equation*}
\end{proof}

Just as in the deterministic case, the derivative loss in~\Cref{thm:L6-strichartz} cannot be avoided.  In fact, by using Plancherel's identity in the space and time variables, we see that
\begin{equation*}
    \lVert e^{-it\Delta} u_0 \rVert_{L^6_{T,x}} = 2\pi  T \sum_{k,j} \lrabs{ \sum_{(k_1,k_2) \in S_{k,j}} a_{k_1} a_{k_2} a_{k-k_1 - k_2}}^2
\end{equation*}
This is is exactly the resonant term in~\eqref{eqn:L6-expansion}.  Thus, if we take initial data $u_0 = N^{-1/2}\sum_{k=0}^N e^{ikx}$, then we find that
\begin{equation*}
    \lVert e^{-iW_t \Delta} u_0 \rVert_{L^6_{\omega,T,x}} \geq \lVert e^{-it\Delta} u_0 \rVert_{L^6_{T,x}} \gtrsim (\log N)^{1/6}
\end{equation*}
where the last inequality comes from~\cite{bourgainFourierTransformRestriction1993a}.

\section{Global wellposedness for \texorpdfstring{$p \leq 3$}{p <= 3}\label{sec:gwp}}

Using the Strichartz estimates given in~\Cref{thm:strichartz-est-homog,thm:strichartz-est-inhomog}, we can prove global wellposedness for~\cref{eqn:WNDNLS} with $1 < p \leq 3$.  We begin by proving that the following truncated equation is locally wellposed:
\begin{equation}\label{eqn:trunc-WNDNLS}
        \left\{\begin{array}{l}
                id\utrunc = \Delta \utrunc \circ dW_t + \Theta[\utrunc](t) \abs{\utrunc}^{p-1} \utrunc\;dt\\
                \utrunc(t=0) = u_0
        \end{array}\right.
\end{equation}
where $\Theta[u](t) = \theta(\lVert \bbOne_{s < t}u(s) \rVert_{L^{4}_{s,x}})$ for a smooth cut-off function $\theta$ satisfying $\theta(x) = 1$ for $\abs{x} \leq 1$ and $\theta(x) = 0$ for $\abs{x} \geq 2$. 
\begin{prop}\label{thm:lwp-trunc-WNDNLS}
    If $u_0 \in L^4_\omega L^2_x$, then for $T$ sufficiently small depending only on $\lVert u_0 \rVert_{L^4_\omega L^2_x}$,~\eqref{eqn:trunc-WNDNLS} has a unique solution in $L^4_{\omega,T,x}$.  This solution is almost surely in $C_T^0 L^2_x$ with conserved $L^2$ norm.  Moreover, if $$\tau = \inf\{t > 0  : \lVert \bbOne_{s < t}\utrunc(s) \rVert_{L^{4}_{s,x}}  = 1\}$$ then  $P(\tau < T) \leq \frac{1}{2}$.
\end{prop}
\begin{proof}
    Let $\mathcal{T}(u) = e^{-iW_t \Delta} u_0 - i \int_0^t e^{-i(W_t - W_s)\Delta} \theta(\lVert \bbOne_{r < s}u(r) \rVert_{L^{2(p-1)}_{r,x}} \abs{u}^{p-1}u\;ds$, and fix a constant $\alpha \in (\frac{1}{2}, 1]$.  We will show that for $T$ sufficiently small, $\mathcal{T}$ is a contraction on $L^4_{\omega,T,x}$.  Using the Strichartz estimates for the linear equation and noting that the truncation effectively allows us to control $u$ in $L^\infty_{\omega}L^4_{T,x}$, we find that
    \begin{equation*}\begin{split}
        \lVert \mathcal{T}(u) \rVert_{L^4_{\omega,T,x}} \lesssim& \left(T + T^{1 - \frac{\alpha}{2}}\right)^{\frac{1}{4}} \lVert u_0 \rVert_{L^4_\omega L^2_x} + \left(T^2 + T^{2 - \alpha}\right)^{\frac{1}{4}} \lVert \Theta[u](t) \abs{u}^{p-1}u \rVert_{L^4_\omega L^{4/3}_{T,x}}\\
        \lesssim& \left(T + T^{1 - \frac{\alpha}{2}}\right)^{\frac{1}{4}} \lVert u_0 \rVert_{L^4_\omega L^2_x} + \left(T^2 + T^{2 - \alpha}\right)^{\frac{1}{4}} T^{\frac{1}{4}\left(1 - \frac{3}{p}\right)} \lVert u \rVert_{L^4_\omega L^{4}_{\omega,T,x}}\\
    \end{split}\end{equation*}
    and
    \begin{equation*}\begin{split}
        \lVert \mathcal{T}(u) - \mathcal{T}(v) \rVert_{L^4_{\omega,T,x}} \lesssim& \left(T^2 + T^{2 - \alpha}\right)^{1/4} \lVert \Theta[u](t) \abs{u}^{p-1}u - \Theta[v](t) \abs{v}^{p-1}v \rVert_{L^4_\omega L^{4/3}_{T,x}}\\
        \lesssim& \left(T^2 + T^{2 - \alpha}\right)^{1/4} T^{\frac{1}{4}\left( 1 - \frac{3}{p}\right)} \lVert u - v \rVert_{L^4_{\omega,T,x}}
    \end{split}\end{equation*}
    once we use~\cite[Lemma 3.3]{debouardStochasticNonlinearSchrodinger1999}) to control the difference.  A fixed point argument now gives us a local solution $\utrunc \in L^4_{\omega,T,x}$ to~\eqref{eqn:trunc-WNDNLS} for $T$ sufficiently small depending only on $L^4_\omega L^2_x$ norm of the initial data.  A standard Fourier truncation argument shows that these fixed point solutions almost surely conserve the $L^2$ mass, and are continuous in $L^2$ a.s. 
    
    It only remains to show that $\Pr(\tau < T) \leq \frac{1}{2}$.  From the Duhamel form of the equation, we see that
    \begin{equation*}\begin{split}
        \E\lVert \utrunc \bbOne_{t \leq \tau} \rVert_{L^4_{T,x}}^4 \leq& C\E\lVert e^{-iW_t\Delta} u_0 \rVert_{L^4_{T,x}}^4\\
        &+ C\E\left\lVert \int_0^t e^{-i(W_t - W_s)\Delta} \abs{\utrunc}^{p-1} \utrunc(s) \bbOne_{s \leq \tau} \;ds\right\rVert_{L^4_{T,x}}^4\\
        \leq& C\left( T + T^{1-\frac{\alpha}{2}}\right) \lVert u_0 \rVert_{L^4_{\omega}L^2_{x}}^4\\
        &+ C\left( T^2 + T^{2- \alpha}\right) T^{1 - \frac{3}{p}} \E \lVert \utrunc \bbOne_{t \leq \tau} \rVert_{L^4_{T,x}}^4
    \end{split}\end{equation*}
    Thus, for $T$ sufficiently small depending only on $\lVert u_0 \rVert_{L^4_\omega L^2_x}$, we have that
    \begin{equation*}
        \E\lVert \utrunc \bbOne_{t \leq \tau} \rVert_{L^4_{T,x}}^4 \leq \frac{1}{2}
    \end{equation*}
    which yields the desired bound by Markov's inequality.
\end{proof}

Since the solution to the truncated equation~\eqref{eqn:trunc-WNDNLS} also solves~\eqref{eqn:WNDNLS} up to the stopping time $\tau$, we can iterate the argument using the $L^2$ conservation to construct a solution $u$ to~\eqref{eqn:WNDNLS}.  The bound $\Pr(\tau < T) \leq 1/2$ guarantees this solution is global.  Moreover, the global solution $u$ is continuous in $L^2$ with conserved mass a.s., and it is the unique solution to~\eqref{eqn:WNDNLS} in $L^4_{\omega}L^4_{t,\text{loc}}L^4_x$.

Finally, we prove that the solutions we constructed in~\Cref{sec:gwp} to the cubic equation agree with the solutions constructed by Chouk and Gubinelli in~\cite{choukNonlinearPDEsModulated2015} using rough path theory.  We recall their result:
\begin{theorem}[Chouk-Gubinelli{~\cite{choukNonlinearPDEsModulated2015}}]
    If $W_t$ is a Brownian motion, then for $p = 3$~\eqref{eqn:WNDNLS} a.s. has a unique solution $v$ (in the rough path sense) with $\psi = e^{iW_t\Delta} v \in C^{1/2}_tL^2_x$.
\end{theorem}
As explained in the introduction,~\cite{choukNonlinearPDEsModulated2014} proves this result by reasoning path-by-path, so it is not immediately clear what moment inequalities the rough path solutions $v$ satisfy.  Examining their fixed point~\cite[Theorem 2.4, Propositions 3.2 and Lemma 3.5]{choukNonlinearPDEsModulated2015}, we see that for a fixed Brownian path $W$ the time of existence for their local solution $\psi^W \in C^{1/2}_tL^2_x$ scales like some negative power of the irregularity constant $\left\lVert \Phi^{W} \right\rVert_{\mathcal{W}_T^{\rho,\gamma}}$ for the path.  Thus, iterating the local construction until we reach time $T$, we see that
\begin{equation*}
    \lVert \psi^{W} \rVert_{C^{1/2}_T L^2_x} \lesssim T \lVert u_0 \rVert_{L^2} \left\lVert \Phi^{W} \right\rVert_{\mathcal{W}_T^{\rho,\gamma}}^r
\end{equation*}
for some positive power $r$.  On the other hand, by~\cite[Remark 4.5]{catellierAveragingIrregularCurves2016}, the irregularity constant for Brownian paths is exponentially square integrable, so $\psi \in L^4_\omega C^{1/2}_T L^2_x$.  It follows immediately that $v \in X^{0, 1/2-}_4(T)$, and since this space embeds in $L^4_{\omega,T,x}$, we have that $u = v$ by uniqueness.

\section{\label{sec:quintic-ill-posedness}\texorpdfstring{$L^2$}{L2} illposedness for \texorpdfstring{$p=5$}{p = 5}}

We now prove~\Cref{thm:quintic-illposedness}.  The $5$th derivative of the data to solution operator $S$ for the quintic equations is $D^5S(0)[\phi,\phi,\phi,\phi,\phi] = i5!\mathcal{A}[\phi]$, were $\mathcal{A}[\phi]$ is the operator
\begin{equation}
    \mathcal{A}[\phi](t) = \int_0^t e^{-i(W_t-W_s)\Delta} \abs{e^{-i(W_s)\Delta} \phi}^4 e^{-i(W_s)\Delta}\phi \;ds
\end{equation}
We will show that there exists a sequence $t_N \to 0$ and $\phi_N$ with $\lVert \phi_N\rVert_{L^2} = O(1)$ such that
\begin{equation*}
    \lim_{N \to \infty} \lVert \E e^{iW_t\Delta} \mathcal{A}[\phi_N](t_N)] \rVert_{L^2} = +\infty
\end{equation*}
In particular, using Minkowski's inequality and the fact that $e^{iW_t\Delta}$ is unitary, we see that
\begin{equation*}
    \E \lVert \mathcal{A}[\phi_N](t_N) \rVert_{L^2} \to \infty
\end{equation*}
so the data to solution map $S : u_0 \mapsto u$ is not $C^5$ when considered as a map $L^2 \to L^1(\Omega; C([0,T]; L^2_x))$ for any $T > 0$.  The illposedness follows from the following result:
\begin{prop}\label{prop:quintic-illposedness-example}
    Define $\phi_N = \frac{1}{N^{1/2}} \sum_{\abs{k} \leq N} e^{ikx}$.  Then, $\lVert\E \mathcal{A}[\phi_N](t) \rVert_{L^2} \gtrsim t \sqrt{\log N}$
\end{prop}
\begin{proof}
    Let $\bbZ_N = \{k \in \bbZ : \abs{k} \leq N\}$.  We calculate
    \begin{equation*}
        e^{-iW_tk^2}\widehat{\mathcal{A}}[\phi_N](t,k) = N^{-5/2} \int_0^t \sum_{\substack{\kappa \in \bbZ_N^5\\ \kappa_1 - \kappa_2 + \kappa_3 - \kappa_4 + \kappa_5 = k}} e^{iW_s\Omega(k, \kappa)}\;ds
    \end{equation*}
    where $\Omega(k,\kappa) = \kappa_1^2 - \kappa_2^2 + \kappa_3^2 - \kappa_4^2 + \kappa_5^2 - k^2$.  Taking the expectation, we have
    \begin{equation}\label{eqn:E-A}\begin{split}
        \E e^{-iW_tk^2}\widehat{\mathcal{A}}[\phi_N](t,k) =& N^{-5/2}\int_0^t \sum_{\substack{\kappa \in \bbZ_N^5\\ \kappa_1 - \kappa_2 + \kappa_3 - \kappa_4 + \kappa_5 = k}} e^{-s\Omega(k, \kappa)^2/2}\;ds\\
        \geq& tN^{-5/2} \abs{S_{N}(k)}
    \end{split}\end{equation}
    where $S_{N}(k) = \{\kappa \in \bbZ_N^5 : \Omega(k, \kappa) = 0, \kappa_1 - \kappa_2 + \kappa_3 - \kappa_4 + \kappa_5 = k\}$.  Writing $\kappa$ as
    \begin{equation}\label{eqn:kappa-reparam}
        \kappa = (k - n_1, k - n_1 - n_2, k - n_2 - n_3, k - n_2 - n_3 - n_4, k - n_2 - n_4)
    \end{equation}
    we see that $\kappa_1 - \kappa_2 + \kappa_3 - \kappa_4 + \kappa_5 = k$ and 
    \begin{equation*}
        \Omega(k, \kappa) = 2 (n_1 n_2 + n_3 n_4)
    \end{equation*}
    Now, for $\abs{k} \leq N/4$, any choice of $\abs{n_i} \leq N/4$ gives $\kappa_i \in \bbZ_N^5$.  Thus, for $\abs{k} \leq N/4$, the cardinality of $S_N(k)$ is bounded below by the number of points $(n_1,n_2,n_3,n_4) \in \bbZ_{N/4}^4$ such that $n_1n_2 + n_3n_4 = 0$.  By~\cite[Lemma 2.1]{kishimotoRemarkPeriodicMass2014}, there are at least $\sim N^2 \log N$ such points, so
    \begin{equation*}
        \lVert \E \widehat{e^{iW_t\Delta}\mathcal{A}}[\phi_N](t,k) \rVert_{L^2} \geq tN^{-5/2} \left(\sum_{\abs{k} \leq N/4} \abs{S_N(k)}^2\right)^{1/2} \gtrsim t\log N
    \end{equation*}
    as desired.
\end{proof}

By Jensen's inequality, $\E \lVert \mathcal{A}[\phi_N](t_N) \rVert_{L^2}$ blows up for $t_N = o((\log N)^{-1})$, which completes the proof of~\Cref{thm:quintic-illposedness}.

\bibliographystyle{abbrv}
\bibliography{refs}
\end{document}